\documentclass[reqno, 11pt]{amsart}
\usepackage{fullpage}
\usepackage{amsfonts}
\usepackage{amssymb}
\usepackage[centertags]{amsmath}
\usepackage{amsthm}
\usepackage{graphicx}
\usepackage{color}
\usepackage{pinlabel}
\usepackage{esint}

\theoremstyle{plain}
\newtheorem{theorem}{Theorem}[section]

\newtheorem{lemma}[theorem]{Lemma}

\theoremstyle{definition}

\theoremstyle{remark}
\newtheorem{remark}[theorem]{Remark}
\numberwithin{equation}{section}

\newcommand{\ep}{\varepsilon}
\renewcommand{\geq}{\geqslant}
\renewcommand{\leq}{\leqslant}
\DeclareMathOperator{\curl}{\mathrm{curl}} 
\renewcommand{\div}{\operatorname{div}}

\newcommand{\bb}{\mathbf b}

\newcommand{\be}{\mathbf e}

\newcommand{\bF}{\mathbf F}

\newcommand{\bh}{\mathbf h}
\newcommand{\bI}{\mathbf I}
\newcommand{\bj}{\mathbf j}
\newcommand{\bJ}{\mathbf J}
\newcommand{\bk}{\mathbf k}
\newcommand{\bL}{\mathbf L}

\newcommand{\bn}{\mathbf n}

\newcommand{\bv}{\mathbf v}

\newcommand{\bx}{\mathbf x}

\newcommand{\by}{\mathbf y}
\newcommand{\bz}{\mathbf z}

\newcommand{\Dvl}{D^{\bf v}_{\ell}}

\newcommand{\cB}{\mathcal B}

\newcommand{\cZ}{\mathcal Z}

\newcommand{\de}{\mathrm{d}}
\newcommand{\dx}{\, \de{\bf x}}

\newcommand{\R}[1]{\mathbb{R}^{#1}}
\newcommand{\RR}{\mathbb{R}}

\renewcommand{\vec}[2]{\left(\begin{array}{c}
		  #1 \\
		  #2
		  \end{array}\right)}

\newcommand{\wtbh}{\widetilde{ {\bf h} } }

\newcommand{\mygraphic}[1]{\includegraphics[height=#1]{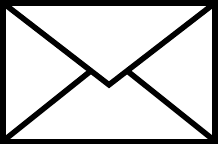}}
\newcommand{\myenv}{(\raisebox{0pt}{\mygraphic{.6em}})}

\makeindex
\begin{document}
\title[Renormalized Energy and Peach-K\"ohler Forces]{Renormalized Energy 
and Peach-K\"ohler Forces
for  Screw Dislocations with Antiplane Shear}

\author{Timothy Blass}
\address{University of California, Riverside, Department of Mathematics, 900 University Ave.\@ Riverside, CA 92521, USA}
\email[T.~Blass]{timothy.blass@ucr.edu} 
\author{Marco Morandotti}
\address{SISSA - International School for Advanced Studies, Via Bonomea, 265, 34136 Trieste, Italy}
\email[M.~Morandotti \myenv]{marco.morandotti@sissa.it}

\begin{abstract}
We present a variational framework for studying screw dislocations 
subject to antiplane shear.
Using a classical model developed by Cermelli \& Gurtin \cite{Gurtin}, 
methods of Calculus of Variations are exploited to prove existence
 of solutions, and to derive a useful expression of the 
Peach-K\"ohler forces acting on a system of dislocation.
This provides a setting for studying the dynamics of the dislocations,
which is done in \cite{BFLM}.
\end{abstract}

\maketitle


\section{Introduction}
Dislocations are one-dimensional defects in crystalline 
materials \cite{N67}. Their modeling 
is of great interest in materials science since 
important material properties, such as rigidity and 
conductivity, can be strongly
affected by
the presence of dislocations. For example, large collections of dislocations can 
result in plastic deformations in solids under applied loads. 

In this paper we derive an expression for the renormalized energy
associated to
a system screw dislocations in
cylindrical crystalline materials using a continuum model 
introduced by Cermelli and Gurtin
\cite{Gurtin}.  
We use the renormalized energy to derive a characterization for the forces
on the dislocations, called Peach-K\"ohler forces. These forces drive
the dynamics of the system, which is studied in \cite{BFLM}. 
The proofs of some results that are used in \cite{BFLM} are contained
in this paper.

Following \cite{Gurtin}, we consider an 
elastic body $B\subset \RR^3$, 
$B := \Omega \times \RR$, where
$\Omega \subset \RR^2$ is a bounded simply
 connected open set with Lipschitz boundary. 
$B$ undergoes antiplane shear
deformations $\Phi:B\to B$  of the form
\begin{equation*}
  \Phi(x_1,x_2,x_3) := (x_1,x_2,x_3 + u(x_1,x_2)),
\end{equation*}
with $u:\Omega\to\RR$. 
The deformation gradient ${\bf F}$ is given by
\begin{equation}\label{deform_grad}
  {\bf F} := \nabla \Phi =\left(
  \begin{array}{ccc}
    1 & 0 & 0\\ 0 & 1 & 0\\ \frac{\partial}{\partial x_{1}}u & \frac\partial{\partial x_{2}}u & 1
  \end{array}
\right) = {\bf I} + {\bf e}_3{\otimes }\left(
    \begin{array}{c}
      \nabla u \\ 0
    \end{array}
\right).
\end{equation}
The assumption of antiplane shear allows us to reduce the three-dimensional problem
to a two-dimensional problem. We will consider strain fields $\bh$
that are defined on the cross-section
$\Omega$, taking values in $\RR^2$. 
In the absence of dislocations,
$\bh = \nabla u$.
If dislocations are present, then the strain field
is singular at the sites of the dislocations, and in the case of screw dislocations this will be a line singularity.

A screw dislocation is a lattice defect at the atomic scale of the
  material, and is represented at the continuum level  by a line singularity in the strain field
for the body $B$. In the antiplane shear setting, this line is parallel
to the $x_3$ axis; in the cross-section $\Omega$ 
a screw dislocation 
is represented as a point singularity.
A screw dislocation is characterized by a position $\bz \in \Omega$ and a vector
$\bb\in \RR^3$, called the \emph{Burgers} vector. The position $\bz \in \Omega$
is a point where the strain field fails to be the gradient of a smooth function, and
the Burgers vector measures the severity of this failure.  To be precise,
a strain field, ${\bf h}$, associated with a system of $N$ screw dislocations at 
positions 
\begin{equation*}\label{Z}
\cZ:=\{\bz_1,\ldots,\bz_N\}
\end{equation*} 
with corresponding Burgers vectors
\begin{equation*}\label{B}
\cB:=\{\bb_1,\ldots,\bb_N\}
\end{equation*}
 satisfies the relation
\begin{equation}
  \label{curl}
  \curl\bh=\sum_{i=1}^N b_i\delta_{\bz_i}\quad {\rm in}\,\, \Omega
\end{equation}
in the sense of distributions, with $b_i:= |{\bf b}_i|$. The notation $\curl\bh$ denotes
the scalar curl,
$\frac{\partial}{\partial x_{1}}h_2 - \frac{\partial}{\partial x_{2}}h_1$.
Thus, in the antiplane shear setting,
the Burgers vectors can be written as
$\bb_i = b_i {\bf e}_3$. The scalar $b_i$ is
called the \emph{Burgers modulus}
for the dislocation at $\bz_i$, and in view of \eqref{curl} it is 
given by
  \begin{equation*}
    b_i=\int_{\ell_i}   {\bf h} \cdot {\bf t}\,\de s,
  \end{equation*}
where $\ell_i$ is any counterclockwise loop surrounding the
dislocation point ${\bf z}_i$ and no other dislocation points, 
${\bf t}$ is the tangent to $\ell_i$,
and $\de s$ is the line element.
Since $\bb_i=b_i\be_3$ for all $i\in\{1,\ldots,N\}$, by abuse of notation
from now on we will use the symbol $\cB$ both for the set of Burgers vectors and for the set of Burgers moduli.
When dislocations are present, 
the deformation gradient $\bF$ can no longer be represented by
the last expression in \eqref{deform_grad}, which 
needs to be replaced with
\begin{equation*}\label{sing_grad}
\bF=\bI+\be_3{\otimes}\vec{\bh}{0}.
\end{equation*}

Our goal is to
derive an energy associated to systems of screw dislocation
and obtain the characterization of the Peach-K\"ohler forces on the dislocations.
This, together with the energy dissipation criterion described in \cite{Gurtin}, will lead to an evolution equation for the system of dislocations.

Our investigation of the energy associated to a system of dislocations 
will be undertaken in the context of linear elasticity for singular strains $\bh$. 
The energy density $W$ 
is given by
\begin{equation*}\label{W_def}
  W({\bf h}) := \frac12 {\bf h} \cdot \bL\bh
\end{equation*}
where the  
elasticity tensor ${\bf L}$ is
a symmetric, positive-definite matrix and, in suitable coordinates,
${\bf L}$ is written in terms of the 
Lam\'e moduli $\lambda,\mu$ of the material as
\begin{equation}\label{L}
  {\bf L} := \left(
    \begin{array}{cc}
      \mu & 0 \\ 0 & \mu \lambda^2
    \end{array}\right).
\end{equation}
We require $ \lambda, \, \mu >0$, and 
the energy is isotropic if and only if $\lambda = 1$. 
The energy of a strain field $\bh$ is given by
\begin{equation*}\label{J_def}
J(\bh):=\int_\Omega W(\bh(\bx))\, \de\bx,
\end{equation*}
and the equilibrium equation is
\begin{equation}
  \label{equilibrium}
  	\div\bL\bh=0 \quad {\rm in}\,\, \Omega.
\end{equation}
Equations \eqref{curl} and \eqref{equilibrium}
provide  a characterization of strain fields describing screw dislocation systems
in linearly elastic materials.
To be precise, we say that  a strain field $\bh \in L^2(\Omega; \R2)$ corresponds to
a {\emph {system of dislocations}} at the positions $\cZ$ with Burgers vectors $\cB$
 if $\bh$ satisfies 
\begin{equation}\label{curldiv}
        \left\{
          \begin{array}{l}
\curl\bh=\sum_{i=1}^N b_i\delta_{\bz_i}\\
	\div\bL\bh=0 
\end{array}\right.\quad {\rm in}\,\, \Omega,
\end{equation}
in the  sense of distributions.

In analogy to the theory of Ginzburg-{L}andau vortices \cite{BBH}, 
 no variational principle can be associated with \eqref{curldiv} 
because the elastic energy of a system of screw dislocations is not finite (see, e.g., \cite{Leoni, Gurtin, N67}),
therefore the study of \eqref{curldiv} cannot be undertaken directly in terms of energy minimization.
Indeed, the simultaneous requirements of finite energy and \eqref{curl} are incompatible,
since if $\curl \bh=\delta_{\bz_0}$, $\bz_0\in\Omega$, 
and if $B_\ep(\bz_0)\subset\subset\Omega$, then
\begin{equation*}
\int_{\Omega \setminus B_\ep(\bz_0)} W(\bh) \,
\de \bx = O(|\log \ep|).
\end{equation*}
In the engineering literature (see, e.g.,  \cite{Gurtin, N67}), this problem is usually overcome by regularizing the energy. 
By removing small cores of size $\ep>0$ centered at the dislocations, we 
will replace $J$ by $J_\ep$ (see \eqref{eq:Jepsilon}) 
and obtain finite-energy strains $\bh_\ep$, as minimizers of $J_\ep$. 
Letting
$\ep\to 0$ we will recover a unique limiting strain $\bh_0 = \lim_{\ep\to 0} \bh_\ep$, 
satisfying \eqref{curldiv}. 
From this, we can derive a \emph{renormalized energy} $U$ associated with
the limiting strain, see \eqref{energyep} and \eqref{U}.
The energy of a minimizing strain takes the form
\begin{equation}\label{expansion}
J_\ep (\bh_\ep) = C \log \frac1\ep + U(\bz_1,\ldots,\bz_N) + O(\ep),
\end{equation}
where the first term, $C\log (1/\ep)$, is the \emph{core energy}, and the
renormalized energy, $U$, is the physically meaningful quantity.
This type of asymptotic expansion was first proved by 
Bethuel, Brezis, and  H{\'e}lein in \cite{BBH93} for  Ginzburg-{L}andau vortices.
The case of edge dislocations was studied in \cite{Leoni}, and also using $\Gamma$-convergence techniques (see, e.g., \cite{AP14,SS03} and the references therein for Ginzburg-{L}andau vortices, \cite{DLGP12,GPPS13,GLP10}). 
Finally, it is important to mention that we ignore here the core
energy. 
We refer  to \cite{N67,TOP96,VKHLO12} for more details.

The renormalized energy $U$ is a function only of the positions $\{\bz_1,\ldots,\bz_N\}$,
and its gradient with respect to $\bz_i$ gives the negative of the Peach-K\"ohler force on $\bz_i$,
denoted $\bj_i$. 
In Theorem. \ref{thm:force} we show that
\[
\bj_i = -\nabla_{\bz_i}U=
\int_{\ell_i} \left\{ W(\bh_0){\bf I} - \bh_0 {\otimes}(\bL\bh_0) \right\}\bn \, \de s,
\]
where $\ell_i$ is a suitably chosen
loop around $\bz_i$ and $\bn$ is the outer unit normal to the set bounded by $\ell_i$ and containing $\bz_i$. 
The quantity $W(\bh_0){\bf I} - \bh_0 {\otimes}(\bL\bh_0)$
is the \emph{Eshelby stress tensor}, see \cite{Eshelby,Gurtin95}.

The expression for $\bj_i$ given below 
contains two contributions accounting for the two different kinds of forces acting on a dislocation when other
dislocations are present:
the interactions with the other dislocations and the interactions with $\partial \Omega$.
The latter balances the tractions
 of the forces generated by all the dislocations, and
it is the (rotated) gradient of the solution $u_0$
to an elliptic problem with Neumann boundary conditions \eqref{u0Neumann}.
Precisely (see \eqref{pkforce}) we show that $\bj_i$ has the form
\begin{equation*}
\bj_i(\bz_1,\ldots,\bz_N)=b_i \bJ \bL\Big[\nabla u_0(\bz_i;\bz_1,\ldots,\bz_N)
+\sum_{j\neq i} \bk_j(\bz_i;\bz_j)\Big],
\end{equation*}
where $\bJ$ is the rotation matrix of an angle $\pi/2$, and
$\bk_j(\cdot;\bz_j)$ is the fundamental singular strain generated by 
the dislocation $\bz_j$ (see \eqref{k_def}).
It is important to notice that the force on the $i$-th dislocation is 
a function of the positions \emph{all} the dislocations.
This explicit formula  
is useful for calculating $\bj_i$, and is employed in
\cite{BFLM} to 
study the motion of the dislocations.

In Section \ref{sec:regularized} we show how to 
regularize the energy to use variational techniques to study the problem.
In Section \ref{sec:renorm} we derive the renormalized energy, which
we use in Section \ref{sec:force} to derive the Peach-K\"{o}hler
force.

\section{Regularized Energies and Singular Strains}\label{sec:regularized}
Consider a system of dislocations at the positions $\cZ = \{\bz_1,\ldots,\bz_N\}$ with 
Burgers vectors $\cB=\{\bb_1,\ldots,\bb_N\}$. 
Regularize the
energy $J$ by removing the singular points from the domain $\Omega$, and define the sets 
\begin{equation}\label{Omega_ep}
\Omega_\ep:=\Omega\setminus\left(\bigcup_{i=1}^N \overline{E}_{\ep,i}\right)\quad
\mbox{for }\,\,\, \ep\in(0,\ep_0),
\end{equation}
where, for every $i\in\{1,\ldots,N\}$, $E_{\ep,i} := E_{\ep}(\bz_i)$, and 
\begin{equation*}
E_r(\bz):=
\left\{ (x_1,x_2)\in\R2: (x_1-z_1)^2 + \left(\frac{x_2-z_2}{\lambda} \right)^2 < r^2 \right\}
\end{equation*}
is an ellipse centered at ${\bf z}$ for $r>0$;
the parameter $\lambda$ is one of the the Lam\'e moduli of the material (cf. \eqref{L}).
Let $\ep_0>0$ be fixed (depending on $\Omega$, $\cZ$, and $\lambda$)
such that for all $\ep \in(0,\ep_0)$ we have $E_{\ep,i} \subset \subset\Omega$, and
$\overline{E}_{\ep,i}\cap \overline{E}_{\ep,j}=\emptyset$ for all $i\neq j$. 
(The shape of the cores $E_{\ep,i}$ is not crucial, but ellipses
 $E_{\ep,i}$ 
centered at $\bz_i$ will be convenient in the sequel.)

We define
\begin{equation}\label{eq:Jepsilon}
  J_\ep (\bh) := \int_{\Omega_\ep} W(\bh)\,\de{\bf x}.
\end{equation}
Note that by removing cores around the singular set $\cZ$, we have regularized the energy
in the sense that it will not necessarily be infinite on strains
satisfying \eqref{curldiv}. However, since we have
effectively removed the dislocations from the problem,
we account for their presence by 
a judicious choice of function space. We define
\begin{equation*}
  H^{\curl}(\Omega_\ep):=\{\bh\in L^2(\Omega_\ep,\R2):\curl\bh\in L^2(\Omega_\ep)\}
\end{equation*}
and
\begin{equation}\label{Hcurl0}
H_0^{\curl}(\Omega_\ep,\cZ,\cB):=\left\{\bh\in H^{\curl}(\Omega_\ep),\,
\curl\bh=0, \, \int_{\partial E_{\ep,i}} \bh{\cdot}{\bf t}\,\de s=b_i,\; i=1,\ldots,N \right\},
\end{equation}
where ${\bf t}$ is the unit tangent vector to $\partial E_{\ep,i}$.
The condition on $\bh$ involving the Burgers moduli $b_i$ in \eqref{Hcurl0}
reintroduces the dislocations into the regularized problem, and 
it 
prevents the minimizers of $J_\ep$ from being
gradients of $H^1$ functions.
In order to abbreviate the notation, 
we will write only $H_0^{\curl}(\Omega_\ep)$ in place of $H_0^{\curl}(\Omega_\ep,\cZ,\cB)$ whenever it is possible to do 
so without confusion.
We will denote by $\bn$ the unit outward normal to $\partial \Omega_\ep$.

The following lemma concerns the properties of minimizers of $J_\ep$, the existence of which
is proved in Lemma~\ref{lem:coercive}. See also Remark \ref{rem:exist}.
\begin{lemma}\label{lem:Euler}
Assume that $\bh_\ep$
is a minimizer of $J_\ep$ in $H_0^{\curl}(\Omega_\ep)$. 
Then it satisfies the Euler equations
  \begin{equation}
    \label{eq:Euler}
\left\{    \begin{array}{ll}
    {\rm div}({\bf L h}_\ep) = 0 & {\rm in}\,\,  \Omega_\ep\,,\\
{\bf L h}_\ep \cdot \bn=0 & {\rm on }\,\, \partial \Omega_\ep\, .
\end{array}\right.
  \end{equation}
Moreover, the solution to \eqref{eq:Euler} is unique.
\end{lemma}
\begin{proof}
Given that the functional $J_\ep$ is quadratic, 
the result is achieved by calculating the vanishing of its
first variation. 
Let $w\in H^1(\Omega_\ep)$; then
\begin{equation*}
\begin{split}
\delta J_\ep(\bh_\ep)[w] &=  \lim_{t\to0} \frac{J_\ep(\bh_\ep+t\nabla w) - J_\ep (\bh_\ep)}{t} \\
&=\lim_{t\to0}\frac1t\left(\int_{\Omega_\ep} t\nabla w \cdot  
{\bf L h}_\ep + \frac12 t^2   \nabla w \cdot
{\bf L} \nabla w \,\de{\bf x}\right) \\
&= \int_{\Omega_\ep}\!\! \nabla w\cdot \bL\bh_\ep\,\de\bx =
-\int_{\Omega_\ep} \!\! w\div(\bL\bh_\ep)\,\de\bx+\int_{\partial\Omega_\ep} \!\!\!\!
w\,\bL\bh_\ep\cdot \bn\,\de s(\bx).
\end{split}
\end{equation*}
By setting $\delta J_\ep(\bh_\ep)[w]=0$ for all $w\in H^1(\Omega_\ep)$, we get \eqref{eq:Euler}.

To prove uniqueness, assume that $\bh_\ep$ and $\widetilde\bh_\ep$ 
both solve system \eqref{eq:Euler}. Then the path integral of the difference
$\bh_\ep-\widetilde\bh_\ep$ over any loop in $\Omega_\ep$ must vanish,
and so $\bh_\ep-\widetilde\bh_\ep = \nabla u$
for some function $u\in H^1(\Omega_\ep)$.
Since $u$ solves the weak Euler equation
\begin{equation*}
\int_{\Omega_\ep} \nabla w \cdot \bL\nabla u\,\de\bx=0,\qquad\text{for all $w\in H^1(\Omega_\ep)$},
\end{equation*}
taking $w=u$ we obtain $J_\ep(\nabla u)=0$, and as $\bL$ is positive definite, we conclude that
$\nabla u=0$.
\end{proof}

\subsection{Singular Strains and the Limit $\ep \to 0$}
We introduce the 
 \emph{singular strains} $\bk_i$ which will be the building blocks of the
singular part of the strain field $\bh$ that represents the
system of dislocations. 
Define $\bk_i(\cdot;\bz_i):\RR^2\setminus\{\bz_i\} \to \R2$, $i=1,\ldots,N$, as
 \begin{equation}
   \label{k_def}
   \bk_i (\bx;\bz_i) = \frac{b_i\lambda}{2\pi(\lambda^2 (x_1-z_{i,1})^2+(x_2-z_{i,2})^2)}
\left(
  \begin{array}{r}
    -(x_2-z_{i,2})\\ x_1-z_{i,1}\,\,
  \end{array}
\right).
 \end{equation}
We will often abbreviate $\bk_i(\cdot ; \bz_i)$ as $\bk_i$. 
Each $\bk_i$ 
can be written as the gradient of a multi-valued function, precisely
\begin{equation*}
  \bk_i (\bx;\bz_i) = \frac{b_i}{2\pi}\nabla_{\bx} \arctan \left(
\frac{x_2-z_{i,2}}{\lambda (x_1-z_{i,1})}\right),
\end{equation*}
and it is straightforward to calculate directly that
  \begin{subequations}
\begin{align}
\curl_\bx\bk_i(\bx;\bz_i)&= b_i\delta_{\bz_i}(\bx)\quad {\rm in}\,\, \RR^2, \label{curlk}\\
{\rm div}_{\bx}\left(\bL \bk_i(\bx;\bz_i)\right)&=0\qquad\qquad\, {\rm in}\,\, \RR^2\setminus \{\bz_i\},
\label{divk}\\
\bL\bk_i (\bx;\bz_i)\cdot \bn&=0\qquad\qquad\, {\rm on}\,\, \partial E_{\ep,i}.\label{bdryk}
\end{align}
  \end{subequations}
In particular, by \eqref{bdryk} and \eqref{divk},
\begin{equation}\label{compatibility}
  \int_{\partial \Omega} \bL\sum_{i=1}^N\bk_i(\by ;\bz_i) \cdot \bn(\by)\, \de s(\by)=
  \int_{\partial\Omega_\ep} \bL\sum_{i=1}^N\bk_i(\by ;\bz_i) \cdot \bn(\by)\, \de s(\by)=0.
\end{equation}
Note that the integral in \eqref{compatibility}
is only well-defined when the dislocations are away from the boundary ($\ep_0>0$).

\begin{lemma} \label{lem:Iep}
Let $\ep_0>0$ be fixed as in \eqref{Omega_ep}. 
For every $\ep \in (0,\ep_0)$, let $\bh \in H_0^{\curl}(\Omega_\ep,\, \cZ,\, \cB )$. Then
\begin{equation}
\label{h_form}
    {\bf h} = \sum_{i=1}^N{\bf k}_i+\nabla u
\end{equation}
for some $u \in H^1(\Omega_\ep)$.
Moreover, the minimization problem
  \begin{equation}\label{minJ}
    \min \left\{ J_\ep(\bh) \,\, \Big| \,\, \bh \in 
H_0^{\curl}(\Omega_\ep,\, \cZ,\, \cB )\right\}
  \end{equation}
is equivalent to the minimization problem
\begin{equation}\label{minI}
  \min \left\{ I_\ep( u) \,\, \Big| \,\, u \in H^1(\Omega_\ep),\; \int_{\Omega_{\ep_0}}u(\bx) \de \bx = 0\right\},
\end{equation}
where
\begin{equation}\label{Iep}
   I_\ep(u) = \int_{\Omega_\ep}W(\nabla u)\de \bx+\sum_{i=1}^N\int_{\partial \Omega}
u  {\bf L k}_i\cdot \bn \,\de s
-\sum_{i=1}^N\sum_{j\neq i}\int_{\partial E_{\ep,i}}  u {\bf L  k}_j\cdot \bn\,\de s
\end{equation}
Minimizers $u_\ep$ of \eqref{minI} are  solutions of the Neumann problem
\begin{equation}\label{uNeumann}
\left\{  \begin{array}{ll}
{\rm div}\left( {\bf L}\nabla u\right) = 0&  {\rm in} \,\, \Omega_\ep, \\
 {\bf L}\!\left(\nabla u+ \sum_{i=1}^N{\bf k}_i \right)\cdot \bn\ = 0&
 {\rm on} \,\, \partial \Omega, \\
  {\bf L}\!\left(\nabla u+ \sum_{j \neq i}{\bf k}_j \right)\cdot \bn\ = 0&
{\rm on} \,\, \partial E_{\ep,i}, \,\, i = 1,2,\ldots, N.
  \end{array}\right.
\end{equation}
\end{lemma}
\begin{proof}
Let $\bh\in H_0^{\curl}(\Omega_\ep,\cZ,\cB)$.
By \eqref{curlk}, $\int_{\ell}(\bh - \sum_{i=1}^N{\bf k}_i )\cdot \de \bx=0$ for
any loop $\ell \subset \Omega_\ep$ and 
thus, $\bh- \sum_{i=1}^N{\bf k}_i=\nabla u$ for some $u\in H^1(\Omega_\ep)$.
In turn
\begin{align}\label{Jexpand}
  J_\ep({\bf h}) = \sum_{i=1}^N J_\ep({\bf k}_i) +
\sum_{i=1}^{N-1}\sum_{j=i+1}^N\int_{\Omega_\ep} {\bf L k}_i\cdot {\bf k}_j\de \bx + I_\ep(u)
\end{align}
where $ I_\ep(u) $ is given by \eqref{Iep} and where in 
the last sum in the expression for $I_\ep$ we omit the terms with $i=j$ because
${\bf L k}_i\cdot \bn=0$ on each $\partial E_{\ep,i}$ (see
\eqref{bdryk}).
(Note that  the last integral in \eqref{Iep} has a 
minus sign because $\bn$ points outside $E_{\ep,i}$, by definition of outer normal).
Hence, the minimization of $J_\ep$ over ${\bf h} \in H_0^{\curl}(\Omega_\ep,\, \cZ,\, \cB)$ 
is achieved by minimizing $I_\ep$ over $u \in H^1(\Omega_\ep)$. 
The normalization condition in \eqref{minI} is introduced in
  order to make the problem coercive, and has the effect of changing
  $u$ in \eqref{h_form} by an additive constant. Since $\nabla u$ is
  the relevant quantity, this does not affect the minimization problem.



To show that minimizers solve the Neumann problem \eqref{uNeumann}, we
calculate the first variation of $I_\ep$ and apply Stokes's theorem to find 
that, given $\varphi\in H^1(\Omega_\ep)$,
\begin{align*}
  \delta I_\ep(u)[\varphi]
 = &   -\int_{\Omega_\ep} \varphi{\rm div} ({\bf L}\nabla u) \,\de x +
\int_{\partial \Omega} \varphi {\bf L} \left(\nabla u +
\sum_{i=1}^N{\bf k}_i \right)\cdot \bn\,\de s 
- \sum_{i=1}^N \int_{\partial E_{\ep,i}}\varphi {\bf L}\left(\nabla u +
\sum_{j \neq i}{\bf k}_j \right)\cdot \bn\,\de s.
\end{align*}
By requiring that $\delta I_\ep(u)[\varphi]=0$ for all $\varphi\in H^1(\Omega_\ep)$, we obtain that
\eqref{uNeumann} is satisfied.
 \end{proof}



The following two lemmas are
slight adaptations of \cite[Lemmas 4.2, 4.3]{Leoni},
so we do not present
the full proofs here.
The key tool is an $\ep$-independent Poincar\'e inequality
for $\Omega_\ep$, \cite[Proposition A.2]{Leoni}. 
\begin{lemma}\label{lem:coercive}
Let $\ep_0>0$ be fixed as in \eqref{Omega_ep}. 
Assume that $\bL$ is positive definite. 
Then there exist positive constants $c_1$ and $c_2$, 
depending only on $\bL$ and $\ep_0$ (in particular, 
independent of $\ep$), such that
\begin{equation}
  \label{Icoercive}
  I_\ep(u) \geq c_1 \|u\|^2_{H^1(\Omega_\ep)}-c_2\|u\|_{H^1(\Omega_\ep)}.
\end{equation}
for all $u \in H^1(\Omega_\ep)$ subject to the constraint
\begin{equation}
  \label{uavg}
  \int_{\Omega_{\ep_0}}u(\bx) \,\de\bx = 0.
\end{equation}
Moreover, for every $\ep \in (0,\ep_0)$ the minimization problem \eqref{minI}
admits a unique solution $u_\ep \in H^1(\Omega_\ep)$
satisfying \eqref{uavg}. Each $u_\ep$ satisfies
\begin{equation}
  \label{ubound}
\| u \|_{ H^1(\Omega_\ep)}\leq M,
\end{equation}
where $M>0$ is a constant independent of $\ep$. 
\end{lemma}
\begin{proof}[Sketch of Proof]
Since $\bL$ is positive definite, we have
\begin{equation*}
  I_\ep(u) \geq  C\int_{\Omega_\ep} |\nabla u|^2\,\de\bx
-\sum_{i=1}^N\sup_{\bx\in\partial \Omega}|\bL \bk_i(\bx, \bz_i)|
\int_{\partial \Omega}|u_\ep|\,\de s
-\sum_{i=1}^N\sum_{j\neq i}\sup_{\bx\in\partial  E_{\ep,i}}|\bL \bk_j(\bx, \bz_j)|
\int_{\partial E_{\ep,i}}|u_\ep|\,\de s
\end{equation*}
Adapting the proof of \cite[Proposition A.2]{Leoni}, for which \eqref{uavg} is crucial, we can find a constant $c_1=c_1(\lambda,\ep_0)$ such that
  \begin{equation}\label{500}
\int_{\Omega_\ep} |\nabla u|^2\,\de\bx \geq c_1\|u\|_{H^1(\Omega_\ep)}^2.
  \end{equation}
Moreover, in \cite{Leoni} it is proved that there exist constants $C_1,C_2$ independent of $\ep$ such that
\begin{equation}\label{bdryLeoni}
  \int_{\partial \Omega}|u_\ep|\,\de s\leq C_1 \|u_\ep\|_{H^1(\Omega_\ep)},
\quad {\rm and}\quad 
\int_{\partial E_{\ep,i}}|u_\ep|\,\de s\leq C_2 \|u_\ep\|_{H^1(\Omega_\ep)}.
\end{equation}
%
%
From the definition of $\bk_i(\bx,\bz_i)$ (see \eqref{k_def}), it is easy to see that there exist constants $c'=c'(\lambda,\ep_0)$ and $c''=c''(\lambda,\ep_0)$ such that
\begin{equation}\label{501}
  \sup_{\bx\in\partial \Omega}|\bL \bk_i(\bx, \bz_i)|<c',
\quad {\rm and}\quad
\sup_{\bx\in\partial  E_{\ep,i}}|\bL \bk_j(\bx, \bz_j)|<c'', \,\,i\ne j.
\end{equation}
Estimates \eqref{500}, \eqref{bdryLeoni}, and \eqref{501} prove \eqref{Icoercive}.
The existence and uniqueness of the solution, and the bound
 \eqref{ubound} are straightforward conclusions from the convexity and coercivity
of the functional $I_\ep$ and the fact that $I_\ep(0)=0$.
\end{proof}

\begin{remark}\label{rem:exist}
  Lemma \ref{lem:Iep} guarantees the equivalence of the minimization
  problems \eqref{minJ} and \eqref{minI}, and Lemma
  \ref{lem:coercive} gives the existence of minimizers for
  \eqref{minI}, thus establishing the existence of minimizers for \eqref{minJ}.
\end{remark}
\begin{lemma}\label{lem:I0}
Assume that $\bL$ is positive definite, and let $u_\ep$ be the
 unique solution to \eqref{minI} that satisfies \eqref{uavg}.
Then, as $\ep\to0$, the sequence $\{u_\ep\}$ converges strongly in
$H_{\rm loc}^1(\Omega\setminus\cZ)$ to a solution $u_0$ of the problem
\begin{equation}\label{minI0}
   \min \left\{ I_0( u) \,\, \Big| \,\, u \in H^1(\Omega),\;
\int_{\Omega_{\ep_0}}u(\bx) \de \bx = 0\right\},
\end{equation}
where
\begin{equation*}
 I_0(u) := \int_{\Omega}W(\nabla u)\,\de\bx+\sum_{i=1}^N\int_{\partial \Omega}
u\,  {\bf L k}_i\cdot {{\bf n}} \,\de s.
\end{equation*}
Moreover,  $I_\ep (u_\ep) \to I_0(u_0)$. 
\end{lemma}
\begin{proof}[Sketch of Proof]
One can extend $u_\ep$ to $\Omega$ and obtain an inequality
$\|u_\ep\|_{H^1(\Omega)}\leq cM$, with $M$ as in \eqref{ubound}
\cite[Prop. A.7]{Leoni}, which leads to $\int_{\partial
  E_{\ep,i}}u_\ep \bL \bk_k \cdot \bn \, \de s\to 0$
as $\ep \to 0$. Also, a subsequence (not relabeled) of 
$\{u_\ep\}$ converges $u_\ep \rightharpoonup u_0$
weakly in $H^1(\Omega)$.
Now, if we fix $\delta \in (0,\ep_0)$ and consider
$\ep <\delta$, from \eqref{Iep} we have
\[
  I_\ep(u_\ep) \geq \int_{\Omega_\delta}
W(\nabla u_\ep)\de \bx+\sum_{i=1}^N\int_{\partial \Omega}
u_\ep  {\bf L k}_i\cdot \bn \,\de s
-\sum_{i=1}^N\sum_{j\neq i}\int_{\partial E_{\ep,i}}  u_\ep {\bf L  k}_j\cdot \bn\,\de s.
\]
Taking $\ep \to 0$ gives $\liminf_{\ep \to 0} I_\ep(u_\ep) \geq
\int_{\Omega_\delta}
W(\nabla u_0)\de \bx+\sum_{i=1}^N\int_{\partial \Omega}
u_0  {\bf L k}_i\cdot \bn \,\de s$.
Taking $\delta \to 0$ gives
$\liminf_{\ep \to 0} I_\ep(u_\ep) \geq I_0(u_0)$. But
$I_\ep(u_\ep)\leq I_\ep(u_0)$, so 
$\limsup_{\ep \to 0} I_\ep(u_\ep) \leq I_0(u_0)$,
and $I_\ep(u_\ep) \to I_0(u_0)$. Strong convergence of $u_\ep \to u_0$
in $H^1(\Omega\setminus \mathcal{Z})$ follows from
convergence of the energies, see \cite{EvansWeak}.
\end{proof}

\begin{remark}\label{rem:Green}
The solutions $u_0$ to \eqref{minI0} are also solutions of the Neumann problem
\begin{equation}\label{u0Neumann}
\left\{  \begin{array}{ll}
{\rm div}\left( {\bf L}\nabla u\right) = 0&  {\rm in} \,\, \Omega, \\
 {\bf L}\!\left(\nabla u+ \sum_{i=1}^N{\bf k}_i \right)\cdot \bn\ = 0&
 {\rm on} \,\, \partial \Omega,
  \end{array}\right.
\end{equation}
and therefore $u_0$
 can be represented in terms
of a Green's function
\begin{equation}
  \label{uGreen}
  u_0(\bx;\bz_1,\ldots,\bz_N) = \int_{\partial \Omega}G(\bx,\by)
\bL\sum_{i=1}^N\bk_i (\by;\bz_i)\cdot \bn(\by)\, \de s(\by),
\end{equation}
exhibiting the explicit dependence on the parameters $\bz_1,\ldots, \bz_N$.
%
 The function  $\nabla u_0(\bx;\bz_1,\ldots,\bz_N)$ represents the elastic strain
at the point $\bx \in \Omega$ due to the presence of $\partial \Omega$
and the dislocations at $\bz_i$ 
with Burgers moduli $b_i$. For this reason, 
we refer to $\nabla u_0(\bx;\bz_1,\ldots,\bz_N)$  as the
boundary-response strain at $\bx$ due to $\cZ$. 
\end{remark}
Combining the results of Lemmas \ref{lem:Iep}, \ref{lem:coercive}, 
and 
\ref{lem:I0}, we conclude the following theorem, which characterizes
the strain field associated to a system of dislocations.

\begin{theorem}\label{thm:limith}
Let $\cZ$ and $\cB$ be given, and let $\Omega\in\RR^2$ be a bounded domain with
Lipschitz boundary $\partial \Omega$. Then the minimization problem
\[
\min_{\bh \in H_0^{\curl} (\Omega_\ep, \cZ, \cB)}\int_{\Omega_\ep}W({\bf h})\de \bx 
\]
admits a unique solution, $\bh_\ep$. 
Moreover, $\bh_\ep \to \bh_0$ strongly in 
$L^2_{\rm loc}(\Omega\setminus \cZ)$, where 
\begin{equation}
  \label{h0def}
  \bh_0(\bx) = \sum_{i=1}^N\bk_i(\bx;\bz_i)+\nabla u_0(\bx;\bz_1,\ldots,\bz_N)
\end{equation}
is a solution of
\begin{equation*}
        \left\{
          \begin{array}{l}
\curl\bh=\sum_{i=1}^N\bb_i\delta_{\bz_i}\\
	\div\bL\bh=0 
\end{array}\right.\quad {\rm in}\,\, \Omega,
\end{equation*}
in the sense of distributions, and
$u_0$ is a minimizer of  \eqref{minI0} and solves 
the Neumann problem \eqref{u0Neumann}.
\end{theorem}

\subsection{Alternative form of the fundamental singular strains}
In the isotropic case, $\lambda = 1$, it can be convenient to use
polar coordinates $(r_i,\theta_i)$ centered at $\bz_i$, rather than
Cartesian coordinates. In the anisotropic case,
when calculating integrals over the cores
$E_{R,i}$, we find some calculations are simplified by using
eccentric anomaly, $\tau_i$, centered at
$\bz_i$, which is defined as
\begin{equation*}
  \tau_i : = \arctan \left(\frac{\tan \theta_i}{\lambda}\right).
\end{equation*}
Using $\tau$, the ellipse $\partial E_{R,i}$ is parametrized by the curve
$\rho(\tau_i) = \bz_i+(R\cos \tau_i, \lambda R\sin \tau_i)$, so
\begin{equation*}
 \bn = \frac{1}{\sqrt{\lambda^2\cos^2 \tau_i + \sin^2 \tau_i\,}} \left(
      \begin{array}{r}
        \lambda \cos \tau_i \\ \sin \tau_i 
      \end{array}\right).
\end{equation*}
For any $\bx \in \Omega$, we can find $r>0$ and $\tau_i$ such that
$\bx=\bz_i+(r\cos\tau_i,\lambda r\sin\tau_i)$.
Substituting the form of $\bx$ into \eqref{k_def} yields
\begin{equation}\label{kanomaly}
\bk_i(\bx;\bz_i) 
= \frac{b_i}{2\pi \lambda r}\left(
      \begin{array}{r}
        -\lambda \sin \tau_i \\ \cos \tau_i 
      \end{array}\right).
\end{equation}



\section{The Renormalized Energy}\label{sec:renorm}
\begin{theorem}\label{thm:3.1}
  Let $0<\ep<\ep_0$ be as in \eqref{Omega_ep} and let $\bh_\ep$ be a solution of \eqref{minJ}. Then
\begin{equation}\label{energyep}
    J_\ep(\bh_\ep) = \int_{\Omega_\ep} \frac12 \bh_\ep \cdot \bL \bh_\ep\, \de \bx
= \sum_{i=1}^N \frac{\mu \lambda b_i^2}{4\pi}\log\frac{1}{\ep} +
U(\bz_1,\ldots,\bz_N) + O(\ep),
  \end{equation}
where
\begin{equation}
  \label{U}
  U(\bz_1,\ldots,\bz_N) :=U_S(\bz_1,\ldots,\bz_N) +U_I(\bz_1,\ldots,\bz_N)
  + U_E(\bz_1,\ldots,\bz_N) 
\end{equation}
and, using \eqref{h0def}, for any $\ep<R<\ep_0$ 
\begin{align}  
&      U_S (\bz_1,\ldots,\bz_N):= \sum_{i=1}^N \frac{\mu \lambda
        b_i^2}{4\pi}\log R+ \sum_{i=1}^N 
      \int_{\Omega\setminus E_{R,i}}W(\bk_i)\dx, \label{US}\\
 &     U_I (\bz_1,\ldots,\bz_N):=\sum_{i=1}^{N-1}  \sum_{j=i+1}^N  
      \int_{\Omega} \bk_j \cdot \bL \bk_i \dx,   \nonumber \\ 
 &     U_E (\bz_1,\ldots,\bz_N):=  \int_{\Omega} W(\nabla u_0)\dx + 
\sum_{i=1}^N \int_{\partial \Omega}u_0 \bL \bk_i \cdot \bn\, \de s. \label{UE}
\end{align}
\end{theorem}
\begin{remark}
We refer to the energy $U$ in \eqref{U} as the \emph{renormalized energy}.
$U_S$ is the 
``self'' energy associated
to the presence of a dislocation, 
$U_I$ is the
energy associated to the interaction between dislocations, and
 $U_E$ is the energy associated to the elastic medium.
 Note that Theorem \ref{thm:3.1} asserts that
the renormalized energy is independent of $\ep$, and we will show that it
can be written in terms of the limit shear $\bh_0$ as in Theorem \ref{thm:limith}. 
This fact
will be used in identifying the force on a dislocation
in Section \ref{sec:force}.
\end{remark}
\begin{proof}
If we expand $J_\ep(\bh_\ep)$ as in \eqref{Jexpand}, we see that
the three terms on the  right side of \eqref{Jexpand} 
correspond to the terms $U_S,\, U_I, \, U_E$. We begin with
$\sum_{i=1}^NJ_\ep(\bk_i)$ and fix $R \in (\ep, \ep_0)$.
Each term in this sum can be written as
\begin{equation*}
  J_\ep(\bk_i)=\int_{\Omega_\ep \setminus E_{R,i}} W(\bk_i)\dx  
\,\, + \,\,  \int_{A_{i,R,\ep}}W(\bk_i)\dx ,
\end{equation*}
where $A_{i,R,\ep} := E_{R,i}\setminus E_{\ep,i}$. Using the
representation for $\bk_i$ in \eqref{kanomaly},
we have
\begin{equation}\label{annulus}
  \int_{A_{i,R,\ep}}\frac12 \bk_i\cdot \bL
  \bk_i \dx 
=\frac{\mu \lambda b_i^2}{4\pi}\log\left(\frac{R}{\ep}\right),
\end{equation}
and this accounts for the 
$\log \frac{1}{\ep}$ term in the energy \eqref{energyep} and the
$\log R$ term in  \eqref{US}.

To show that
\begin{equation}\label{conv}
 \sum_{i=1}^{N-1} \sum_{j=i+1}^N \int_{\Omega_\ep} \bk_j \cdot \bL \bk_i \dx
\longrightarrow
\sum_{i=1}^{N-1} \sum_{j=i+1}^N \int_{\Omega}
 \bk_j \cdot \bL \bk_i \dx \qquad {\rm as}\quad 
\ep \to 0,
\end{equation}
we note that $\bk_i$ is integrable in $E_{R,i}$ (it grows like
$r_i^{-1}$) and $\bL\bk_j$ is bounded on $E_{R,i}$ for $j\neq i$,
hence \eqref{conv} holds by Lebesgue Dominated Convergence Theorem.
From Lemma \ref{lem:I0}, we have that $I_\ep(u_\ep) \to I_0(u_0)$ as
$\ep \to 0$, whence \eqref{UE} follows.

To show that $U$ is independent of $R$, we need only show that $U_S$ 
is independent of $R$. If we take $R' \ne R$, without loss of
generality we can assume $R'<R$, then by \eqref{annulus}
\[
\int_{\Omega \setminus E_{R',i}} W(\bk_i)\dx-\int_{\Omega \setminus E_{R,i}} W(\bk_i)\dx
=\int_{A_{i,R,R'}}W(\bk_i)\dx =\frac{\mu \lambda  b_i^2}{4\pi}\log \frac{R}{R'},
\]
so that
$$
 \int_{\Omega \setminus E_{R',i}} W(\bk_i)\dx  +
\frac{\mu \lambda b_i^2}{4\pi}\log R' 
= \int_{\Omega \setminus E_{R,i}} W(\bk_i)\dx  +
\frac{\mu \lambda b_i^2}{4\pi}\log R,
$$
which shows that \eqref{US} is independent of the choice of $R<\ep_0$.
\end{proof}

The renormalized energy $U$ will blow up like the $\log$ of the
distance between dislocations, i.e. $U \sim -\log|\bz_i-\bz_j|$. This
is made precise in \cite{BFLM}.



\section{The Force on a Dislocation}\label{sec:force}
In this section we determine the force $\bj_i$ on the dislocation at $\bz_i$ for a
given a system of dislocations $\cZ$ with Burgers vectors $\cB$,
and show that $\bj_i = -\nabla_{\bz_i} U$. 
Following \cite{Gurtin}, the Peach-K\"ohler force on the dislocation
at $\bz_i$ (also called the net configurational traction)
is given by
\begin{equation}
  \label{PKdef}
  \bj_i := \lim_{R\to 0}\int_{\partial E_{R,i}}{\bf C}\bn \, \de s,
\end{equation}
where the stress tensor is the
Eshelby  stress (\cite{Eshelby,Gurtin95})
  \begin{equation}
    \label{eshelby}
    {\bf C} := W(\bh_0){\bf I} - \bh_0{\otimes}(\bL\bh_0).
  \end{equation}
Here ${\bf I}$ is the identity matrix and $\bh_0$ is defined in \eqref{h0def}.
\begin{theorem}\label{thm:force}
Let $\bh_0$ be the limiting singular strain
defined by \eqref{h0def} and let $U$ the associated renormalized energy given in \eqref{U}.
Then for $\ell\in\{1,\ldots,N\}$ and any $R\in(0,\ep_0)$
  \begin{equation}
    \label{gradU}
 \nabla_{\bz_\ell}U(\bz_1,\ldots,\bz_N) =
-\int_{\partial E_{R,\ell}} \left\{ W(\bh_0){\bf I} - \bh_0{\otimes}(\bL\bh_0) \right\}\bn \, \de s,
  \end{equation}
and so the force on the
dislocation at $\bz_\ell$ is given by 
\begin{equation}
  \label{jgradU}
\bj_\ell = -\nabla_{\bz_\ell}U.
\end{equation}
Moreover, 
\begin{equation}
  \label{pkforce}
  \bj_\ell (\bz_1,\ldots,\bz_N)
= b_\ell \bJ \bL \left(\nabla u_0(\bz_\ell;\bz_1,\dots,\bz_N) +
\sum_{i\neq \ell}\bk_i(\bz_\ell;\bz_i)\right), \quad
{\rm where}\quad 
\bJ = \left(
  \begin{array}{rr}
    0&1 \\ -1& 0
  \end{array}
\right),
\end{equation}
and $u_0$ is the solution to \eqref{u0Neumann}.
\end{theorem}

\begin{proof}
Formula \eqref{gradU} is proved in the Appendix.
From \eqref{gradU}, we show \eqref{jgradU} and
 \eqref{pkforce} as follows.
Recall that the renormalized energy is independent of $R<\ep_0$
 (see the proof of Theorem \ref{thm:3.1}), so 
\begin{equation}\label{pklim}
 -\nabla_{\bz_\ell}U=  \int_{\partial E_{R,\ell}}{\bf C}\bn\, \de s = 
\lim_{R\to 0}\int_{\partial E_{R,\ell}}{\bf C}\bn\, \de s =\bj_\ell,
\end{equation}
establishing \eqref{jgradU} in view of \eqref{PKdef}.

 The field $\bh_0$ has a singularity at $\bz_\ell$ which comes from
the term $\bk_\ell$ (see \eqref{h0def}), 
and we decompose $\bh_0$ into the \emph{singular part at $\bz_\ell$} and
the \emph{regular part at $\bz_\ell$},
\begin{equation}
  \label{hreg}
  \bh_0(\bx) = \bk_\ell (\bx;\bz_\ell) + \wtbh (\bx),
\qquad {\rm where}\,\,\,\, \wtbh (\bx) := 
\nabla u_0(\bx) +
\sum_{i\neq \ell}\bk_i(\bx;\bz_i).
\end{equation}
Using \eqref{hreg}, we write the Eshelby stress ${\bf C}$ from \eqref{eshelby} as
\begin{equation*}
  {\bf C}  = \left(\frac12 \bk_\ell{\cdot} \bL \bk_\ell +
\bk_\ell{\cdot} \bL \wtbh + \frac12 \wtbh {\cdot} \bL \wtbh \right)\bI
- \bk_\ell{\otimes}(\bL \bk_\ell)-\bk_\ell{\otimes}(\bL \wtbh)
-\wtbh{\otimes}(\bL \bk_\ell)-\wtbh{\otimes}(\bL \wtbh).
\end{equation*}
Since $\wtbh$ is smooth and bounded on $\overline{E}_{R,\ell}$ we have
\begin{equation*}
  \lim_{R\to 0}\int_{\partial E_{R,\ell}} \left(
\frac12 \wtbh \cdot \bL \wtbh \right)  \bn\, \de s = 0
\quad {\rm and} \,\,
\lim_{R\to 0}\int_{\partial E_{R,\ell}}\wtbh{\otimes}(\bL \wtbh) \bn\,  \de s =0.
\end{equation*}
Using the fact that $\bL \bk_\ell \cdot \bn=0$ on $\partial E_{R,\ell}$ (see \eqref{bdryk}) we have
\begin{equation*}
  \int_{\partial E_{R,\ell}}\wtbh {\otimes}\left( \bL \bk_\ell\right)\bn\, \de s=0,
\quad {\rm and} \,\,
\int_{\partial E_{R,\ell}}\bk_\ell{\otimes}\left( \bL \bk_\ell\right)\bn\, \de s=0,
\quad \forall R<\bar{R}.
\end{equation*}
Using \eqref{kanomaly} 
we have
$\bk_\ell \cdot \bL \bk_\ell =  \mu b_\ell^2/(4\pi^2 R^2)$
 on $\partial E_{R,\ell}$, and so
\begin{equation*}
  \int_{\partial E_{R,\ell}} \frac12 (\bk_\ell \cdot \bL \bk_\ell) \bn\, \de s
=  \frac{\mu b_\ell^2}{8\pi^2 R^2}\int_{\partial E_{R,\ell}}  \bn\, \de s = 0,
\end{equation*}
for all $R<\ep_0$.
Therefore the only contribution in \eqref{pklim} will come from 
\begin{align*}
  \left(\left( \bk_\ell \cdot \bL\wtbh \right)\bI -\bk_\ell {\otimes}(\bL \wtbh)\right)\bn
&=\left(\bn{\otimes} \bk_\ell\right)\bL\wtbh  -\left( \bk_\ell {\otimes}\bn \right) \bL \wtbh.
\end{align*}
Now, using \eqref{kanomaly}, it is easy to see that
\begin{equation*}
  \bn{\otimes} \bk_\ell-  \bk_\ell {\otimes}\bn
= \frac{b_\ell}{2\pi \lambda R\sqrt{\lambda^2\cos^2 \tau + \sin^2 \tau\,}}
\left(
  \begin{array}{cc}
    0 & \lambda \\ -\lambda & 0
  \end{array}\right)
\end{equation*}
and, since  $\de s =  R\sqrt{\lambda^2\cos^2 \tau + \sin^2 \tau\,}\de\tau$,
\begin{align*}
  \int_{\partial E_{R,\ell}} \left(\bn{\otimes} \bk_\ell-  \bk_\ell {\otimes} \bn
\right) \bL\wtbh \, \de s= 
\frac{b_\ell}{2\pi}\int_0^{2\pi}  \bJ \bL \wtbh \,\de \tau.
\end{align*}
Since the integrand is smooth on $E_{R,\ell}$, 
we conclude that 
\begin{align*}
  \lim_{R\to 0}\int_{\partial E_{R,\ell}} \left(\bn{\otimes} \bk_\ell-  \bk_\ell{\otimes} \bn
\right) \bL\wtbh \, \de s=
\frac{b_\ell}{2\pi}\int_0^{2\pi}  \bJ \bL \wtbh(\bz_\ell) \,\de \tau 
= b_\ell \bJ \bL \wtbh(\bz_\ell),
\end{align*}
which, in view of \eqref{pklim}, establishes \eqref{pkforce}.
\end{proof}
\begin{remark}\label{rem:pksmooth}
The formula \eqref{pkforce} gives the force on 
the dislocation at $\bz_\ell$, 
and it shows that, as a function of $\bz_\ell$, 
the force $\bj_\ell$ is smooth in the interior
of $\Omega \setminus \{\bz_1,\ldots,\bz_{\ell-1},\bz_{\ell+1},\ldots,\bz_N\}$.
That is, provided $\bz_\ell$ is not colliding with another dislocation or with
$\partial \Omega$, then the force is given by a smooth function.
Of course, $\bj_\ell$ depends on the positions of \emph{all} the dislocations,
and the same reasoning applies to $\bj_\ell$ as a function of any $\bz_i$. 
\end{remark}
\begin{remark}
  We find agreement between
 \eqref{pkforce} and equation (8.18) from \cite{Gurtin},
where the force on $\bz_\ell$ 
is given by $b_\ell$ times a $\pi/2$-rotation of the 
regular part of the strain at $\bz_\ell$ (i.e., $\widetilde \bh$). 
Since we have a formula for the regular
part, we are able to write the Peach-K\"ohler force more
explicitly (in terms of the solution to
\eqref{u0Neumann}). We have also shown that assumption (A3) from 
\cite{Gurtin} holds for screw dislocations, validating 
the derivation of (8.18) in \cite{Gurtin}.
\end{remark}

\section*{Acknowledgements}
The authors warmly thank the Center for Nonlinear Analysis (NSF Grant No.
DMS-0635983), where part of this research was carried out. 
The research of T. Blass was supported by the National Science
Foundation under the PIRE Grant No. OISE-0967140. 
M. Morandotti also acknowledges support of the 
Funda\c{c}\~{a}o para a Ci\^encia e a Tecnologia 
(Portuguese Foundation for Science and Technology)
through the Carnegie Mellon Portugal Program 
under the grant FCT$\_$UTA/CMU/MAT/0005/2009.

\section{Appendix}
We present 
the proof of \eqref{gradU} along with
some necessary lemmas.
We begin by noting that 
$U (\bz_1,\ldots,\bz_N) = \widehat{U} (\bz_1,\ldots,\bz_N)+
\overline{U} (\bz_1,\ldots,\bz_N)$
where
\begin{align}
   \widehat{U}(\bz_1,\ldots,\bz_N) &= 
\int_{\Omega_\ep}W(\bh_0)\,\de\bx, \nonumber \\ 
  \overline{U}(\bz_1,\ldots,\bz_N) &= 
\sum_{i=1}^N\sum_{m\ne i}\int_{E_{\ep,m}}W(\bk_i)\,\de\bx + \sum_{m=1}^N
\sum_{i=1}^{N-1}\sum_{j=i+1}^N \int_{E_{\ep,m}}\bk_j\cdot\bL\bk_i\,\de\bx \,\, + \nonumber \\
& \qquad + \sum_{m=1}^N\int_{E_{\ep,m}} W(\nabla u_0)\,\de\bx + 
\sum_{m=1}^N\sum_{i=1}^N\int_{\partial E_{\ep,m}}u_0 \bL\bk_i\cdot \bn\, \de s,
\label{Ubar}
\end{align}
which follows from a direct calculation and integration by parts to 
eliminate the integral over $\partial\Omega$ from $U_E$.

We introduce the notation $\Dvl u$ for
the derivative of
a function $u = u(\bx;\bz_1,\ldots,\bz_N)$ 
with respect to the $\ell$-th dislocation location in the direction $\bv$,
\begin{equation*}
  \Dvl u(\bx) := \frac{\de}{\de\xi}u(\bx;\bz_1,\ldots,\bz_\ell + 
\xi\bv,\ldots,\bz_N)\Big|_{\xi=0}.
\end{equation*}

\begin{lemma}\label{lem:kdot}
The fields $\bk_i(\bx;\bz_i)$,  $u_0(\bx;\bz_1,\ldots,\bz_N)$,
and $\bh_0(\bx;\bz_1,\ldots,\bz_N)$
 are smooth with 
respect to $\bz_\ell$ for every $\ell\in\{1,\ldots,N\}$. Moreover, 
$\Dvl \bk_i(\bx) = 0$ if $\ell \ne i$,
\begin{align}
&  \Dvl \bk_\ell (\bx)  = -D \bk_\ell(\bx)\bv = 
-\nabla \left(\bk_\ell(\bx)\cdot\bv \right) \label{dvlk}\\
&\Dvl \bh_0 (\bx)=\nabla w(\bx), \quad {\rm where}\quad 
w(\bx) = \Dvl u_0(\bx) - \bk_\ell(\bx)\cdot\bv \label{dvlh}
\end{align}
\end{lemma}

\begin{proof}
The form of $\bk$ in \eqref{k_def} shows that $\bk_i$ is smooth 
with respect to $\bz_\ell$ for all $i,\ell=1,\ldots, N$, and in particular that
  $\bk_i(\bx) = \bk(\bx;\bz_i)$ is independent of $\bz_\ell$ if $\ell \ne i$
so $\Dvl \bk_i =0$. That form also shows
that $\bk(\bx;\bz_\ell + \xi\bv) = \bk(\bx -\xi\bv;\bz_\ell)=\bk_\ell(\bx-\xi\bv)$ so 
that $\Dvl \bk_\ell(\bx) = -(D\bk_\ell )\bv$, where $D\bk_\ell$ 
is the derivative of $\bk_\ell$ 
with respect to $\bx$. Now because $\curl \bk_\ell = 0$, we have $D \bk_\ell(\bx)\bv = 
\nabla \left(\bk_\ell(\bx)\cdot\bv \right)$, which establishes \eqref{dvlk}.

Since $u_0$ solves the elliptic problem \eqref{u0Neumann}, it
can be   represented as in \eqref{uGreen},
in terms of the Green's function $G(\bx,\by)$. The smoothness
of $u_0$ in $\bz_\ell$ follows from the smoothness of $\bk_i$ 
for each $i,\ell =1,\ldots,N$. Hence, $\bh_0$ is smooth in $\bz_\ell$
and $\Dvl \bh_0 = \Dvl \nabla u_0 + \Dvl\bk_\ell = \nabla
(\Dvl u_0 - \bk_\ell \cdot \bv)$, which establishes \eqref{dvlh}.
\end{proof}

We will take derivatives of the energy with respect to the dislocations positions. This
will involve integrals over cores that are centered at $\bz_\ell + \xi \bv$
whose integrands are evaluated on these shifted cores or on their complements
in $\Omega$. Thus, we will need to be able to take derivatives of integrals
over sets that depend on $\xi$ and whose integrands are 
functions that depend on $\xi$. 

\begin{lemma}\label{lem:derivatives}
Let $f=f(\bx,\xi)$, $g=g(\bx,\xi)$, and $r=r(\bx,\xi)$ be defined on $E_\ep(\bx_0+\xi\bv)$, 
$\partial E_\ep(\bx_0+\xi\bv)$, and $\Omega\setminus E_\ep(\bx_0+\xi\bv)$, respectively, for $\xi$ 
a real parameter, $\ep>0$, $\bv\in\R2$. Then

\begin{align}
 \left.\frac{\de}{\de\xi}\int_{E_\ep(\bx_0+\xi\bv)} f(\bx,\xi)\,\de\bx\right|_{\xi=0} 
 &=  \int_{E_\ep(\bx_0)} D_\xi f(\bx,0)\,\de\bx \nonumber\\
& = \int_{E_\ep(\bx_0)} \partial_\xi f(\bx,0)\,\de\bx+\int_{\partial E_\ep(\bx_0)} 
f(\bx,0)\bv\cdot \bn\,\de s, \label{fder}\\
\left.\frac{\de}{\de\xi}\int_{\partial E_\ep(\bx_0+\xi\bv)} 
g(\bx,\xi)\,\de s \right|_{\xi=0} & =  \int_{\partial E_\ep(\bx_0)} 
D_\xi g(\bx,0)\,\de s,\label{gder} \\
\left.\frac{\de}{\de\xi}\int_{\Omega\setminus E_\ep(\bx_0+\xi\bv)} 
r(\bx,\xi)\,\de\bx\right|_{\xi=0} & =  
\int_{\Omega\setminus E_\ep(\bx_0)} \partial_\xi r(\bx,0)\,\de\bx-
\int_{\partial E_\ep(\bx_0)} r(\bx,0)\bv\cdot \bn\,\de s, \label{rder}
\end{align}
where $D_\xi f:=\partial_\xi f+\nabla f\cdot\bv$.
\end{lemma}

\begin{proof}
  We calculate
  \begin{align*}
    \frac{\de}{\de\xi} \int_{E_\ep(\bx_0+\xi\bv)} f(\bx,\xi)\de\bx=
    \frac{\de}{\de\xi} \int_{E_\ep(\bx_0)} f(\bx+\xi\bv,\xi)\de\bx=
     \int_{E_\ep(\bx_0)} (\partial_\xi f(\bx+\xi\bv,\xi)
+\nabla f(\bx+\xi\bv,\xi)\cdot \bv) \de\bx.
  \end{align*}
If we send $\xi \to 0$ and apply the divergence theorem
we obtain \eqref{fder}. A similar calculation gives \eqref{gder}
but the divergence theorem is not applied. If $\hat{r}$ is 
a smooth extension of $r$ to $\Omega$ then
\begin{align*}
  \frac{\de}{\de\xi}\int_{\Omega\setminus E_\ep(\bx_0+\xi\bv)} r(\bx,\xi)\de\bx  &=
\frac{\de}{\de\xi}\int_{\Omega} \hat{r}(\bx,\xi)\de\bx -
\frac{\de}{\de\xi}\int_{E_\ep(\bx_0+\xi\bv)} \hat{r}(\bx,\xi)\de\bx \\ 
&=\int_{\Omega}\partial_\xi \hat{r}(\bx,\xi)\de\bx -
\int_{E_\ep(\bx_0)} \partial_\xi\hat{r}(\bx+\xi\bv,\xi)\de\bx -
\int_{\partial E_\ep(\bx_0)} \hat{r}(\bx+\xi\bv,\xi)\bv\cdot \bn\,\de s.
\end{align*}
Setting $\xi=0$ and combining the first two integrals on the right side yields
\eqref{rder}.
\end{proof}

\begin{remark}
Lemma \ref{lem:derivatives} applies 
to the vector-valued $\bk_i$.
 When applying Lemma \ref{lem:derivatives} to integrals of 
$\bk(\bx;\bz_\ell + \xi\bv)$ over $E_{\ep}(\bz_\ell + \xi\bv)$
we will get cancellations from 
\begin{equation}\label{dkzero}
   D_\xi \bk(\bx;\bz_\ell + \xi\bv)
= \partial_\xi\bk(\bx;\bz_\ell + \xi\bv)+ D\bk(\bx;\bz_\ell + \xi\bv)\bv
= \Dvl \bk_\ell(\bx) + D\bk_\ell \bv = 0.
\end{equation}
The last equality follows from \eqref{dvlk}.
\end{remark}

\begin{proof}[Proof of Equation \eqref{gradU}]
The $-\log \ep$  term in the energy is independent of the positions of
the dislocations so it vanishes upon taking the derivative of the
energy with respect to $\bz_\ell$. To calculate the derivative of $U$ 
with respect to $\bz_\ell$ will split $\nabla_{\bz_\ell}U$ into
$\nabla_{\bz_\ell}\widehat{U}+\nabla_{\bz_\ell}\overline{U}$.
To calculate $\nabla_{\bz_\ell}\widehat{U}$ we apply \eqref{rder} to get
\begin{equation}
  \nabla_{\bz_\ell}\widehat{U}= 
\Dvl \left( \int_{\Omega_\ep}W(\bh_0)\de\bx \right) =
\int_{\Omega_\ep}\Dvl\bh_0 \cdot \bL\bh_0\, \de\bx -
\int_{\partial E_{\ep,\ell}} W(\bh_0) \bv \cdot \bn\, \de s
\label{DUhat1}
\end{equation}
Using \eqref{dvlh}, ${\rm div}(\bL\bh_0)=0$ 
in $\Omega$, and
$\bL \bh_0 \cdot \bn = 0$ on $\partial \Omega$, we have
\begin{equation}\label{DUhat2}
\begin{split}
\int_{\Omega_\ep}\Dvl\bh_0 \cdot \bL\bh_0 \,\de\bx &= \int_{\Omega_\ep}\nabla (
\Dvl u_0 - \bk_\ell \cdot \bv) \cdot \bL\bh_0\, \de\bx = 
\int_{\partial \Omega_\ep}  (\Dvl u_0 - \bk_\ell \cdot \bv)   \bL\bh_0 \cdot \bn\, \de s\\
& = -\sum_{j=1}^N\int_{\partial E_{\ep,j}} (\Dvl u_0 - \bk_\ell \cdot \bv) \bL \bh_0 \cdot \bn\, \de s
= -\sum_{j=1}^N\int_{\partial E_{\ep,j}} w\bL \bh_0 \cdot \bn\, \de s
\end{split}
\end{equation}
using the notation $w = \Dvl u_0 - \bk_\ell \cdot \bv$ from \eqref{dvlh}.
Combining \eqref{DUhat1} and \eqref{DUhat2}, and adding and subtracting
$ \bh_0 \otimes (\bL\bh_0)\bn \cdot \bv$ from the integrand, we have
\begin{equation*}
\begin{split}
\nabla_{\bz_\ell}\widehat{U}&= 
 -\sum_{j=1}^N\int_{\partial E_{\ep,j}} (\Dvl u_0 - \bk_\ell \cdot \bv) \bL \bh_0 \cdot \bn\, \de s
-\int_{\partial E_{\ep,\ell}} W(\bh_0) \bv \cdot \bn\, \de s \\
&= -\int_{\partial E_{\ep,\ell}} \left\{ W(\bh_0){\bf I} - \bh_0
  \otimes (\bL\bh_0) \right\}\bn \cdot \bv \, \de s
-\sum_{j \ne \ell}   \int_{\partial E_{\ep,j}} (\Dvl u_0 - \bk_\ell \cdot \bv) \bL \bh_0 \cdot \bn \, \de s \,\,+\\
&\qquad \qquad \qquad - \int_{\partial E_{\ep,\ell}} [(\Dvl u_0 - \bk_\ell \cdot \bv) \bL \bh_0 \cdot \bn  + 
\bh_0  \otimes (\bL\bh_0) \bn \cdot \bv] \, \de s;
\end{split}
\end{equation*}
also, $\bh_0  \otimes (\bL\bh_0) \bn \cdot \bv = 
(\bh_0 \cdot \bv) (\bL\bh_0 \cdot \bn )= 
(\nabla u_0 \cdot \bv+\sum_{i=1}^N \bk_i \cdot \bv) (\bL\bh_0 \cdot \bn)$, so 
\begin{equation*}
\begin{split}
  (\Dvl u_0 - \bk_\ell \cdot \bv) \bL \bh_0 \cdot \bn   + 
\bh_0  \otimes (\bL\bh_0) \bn \cdot \bv  
&=  \left(\Dvl u_0 - \bk_\ell \cdot \bv  + 
\nabla u_0 \cdot \bv+\sum_{i=1}^N \bk_i \cdot \bv \right)  \bL\bh_0 \cdot \bn \\
& = \left(D_\xi u_0 + \sum_{i\ne \ell}\bk_i \cdot \bv \right) \bL  \bh_0 \cdot \bn
\end{split}
\end{equation*}
where $D_\xi u_0 = \Dvl u_0 +\nabla u_0 \cdot \bv$. Hence,
\begin{equation}
\begin{split}
\nabla_{\bz_\ell}\widehat{U}&= 
-\int_{\partial E_{\ep,\ell}} \left\{ W(\bh_0){\bf I} - \bh_0  \otimes 
(\bL\bh_0) \right\}\bn \cdot \bv \, \de s\\
& -\sum_{j \ne \ell}   \int_{\partial E_{\ep,j}} (\Dvl u_0 - \bk_\ell \cdot \bv) \bL \bh_0 \cdot \bn \, \de s
- \int_{\partial E_{\ep,\ell}}\left(D_\xi u_0 + \sum_{i\ne \ell}\bk_i \cdot \bv \right) \bL  \bh_0 \cdot \bn\, \de s
  \label{DUhat}
\end{split}
\end{equation}

We calculate $\nabla_{\bz_\ell}\overline{U}$ in several steps. 
We split the first sum in the right side of \eqref{Ubar} 
into the integral over $E_{\ep,\ell}$ and the rest of the terms
\begin{equation}\label{split1}
 \sum_{i=1}^N\sum_{m\ne i}\int_{E_{\ep,m}}W(\bk_i)\de\bx
= \int_{E_{\ep,\ell}}  \sum_{m\ne \ell}  W(\bk_m)\de\bx
+\sum_{m\ne \ell} \int_{E_{\ep,m}} \sum_{i\ne m}  W(\bk_i)\de\bx.
\end{equation}
In the first of these, each $\bk_m$ 
 does not vary as $\bz_\ell \to \bz_\ell
+\xi\bv$ because $m\ne \ell$.
 Hence we apply \eqref{fder} directly with
$D_\xi \bk_m = \partial_\xi \bk_m + \Dvl \bk_m
 = \nabla (\bk_m \cdot \bv)$ because $\partial_\xi \bk_m=0$. We have
\begin{equation}
\begin{split}
\Dvl \left(\sum_{m\ne \ell} \int_{E_{\ep,\ell}} W(\bk_m)\,\de\bx \right)
&= \sum_{m\ne \ell} \int_{E_{\ep,\ell}} D_\xi \bk_m \cdot \bL \bk_m\,\de\bx 
 = \sum_{m\ne \ell} \int_{E_{\ep,\ell}}\nabla(\bk_m\cdot \bv) \cdot \bL \bk_m\, \de\bx  \\
& = \int_{\partial E_{\ep,\ell}}   \sum_{m\ne \ell}   (\bk_m \cdot \bv) \bL \bk_m\cdot \bn\, \de s,
\label{term1}
\end{split}
\end{equation}
where we used ${\rm div}(\bL\bk_m)=0$.

The second term from \eqref{split1} involves integrals over $E_{\ep,m}$ 
for $m\ne \ell$, so these domains do not move as $\bz_\ell \to \bz_\ell +\xi\bv$.
Also, the terms $W(\bk_i)$ for $i\neq \ell$ vanish when we apply $\Dvl$, so
\begin{equation}
\begin{split}
\Dvl \left(\sum_{m\ne \ell} \int_{E_{\ep,m}} \sum_{i\ne m}  W(\bk_i)\dx\right)
&= \sum_{m\ne \ell} \int_{E_{\ep,m}} \Dvl \bk_\ell \cdot \bL \bk_\ell\, \de\bx 
 = \sum_{m\ne \ell} \int_{E_{\ep,m}} - \nabla(\bk_\ell\cdot \bv) \cdot \bL \bk_\ell\, \de\bx  \\
& = -\sum_{m\ne \ell} \int_{\partial E_{\ep,m}} (\bk_\ell \cdot \bv) \bL \bk_\ell\cdot \bn\, \de s,
\label{term5}
\end{split}
\end{equation}
where we used \eqref{dvlk} and  ${\rm div}(\bL\bk_\ell)=0$.

The second sum from \eqref{Ubar} is split into 
the integral over $E_{\ep,\ell}$ and the rest of the terms
\begin{equation}
  \label{split2}
   \sum_{m=1}^N\sum_{i=1}^{N-1}\sum_{j=i+1}^N \int_{E_{\ep,m}}\bk_j\cdot\bL\bk_i\,\de\bx
= \int_{E_{\ep,\ell}}  \sum_{i<j}  \bk_j\cdot\bL\bk_i\,\de\bx + 
\sum_{m\ne \ell} \int_{E_{\ep,m}}  \sum_{i<j}  \bk_j\cdot\bL\bk_i\,\de\bx .
\end{equation}
Applying \eqref{fder} to the first term on the right side yields
\begin{equation}
\begin{split}
 \Dvl \left( \int_{E_{\ep,\ell}}  \sum_{i<j}  \bk_j\cdot\bL\bk_i\,\de\bx\right) & = 
\int_{E_{\ep,\ell}}  \sum_{i<j}  D_\xi\bk_j\cdot\bL\bk_i +  D_\xi \bk_i\cdot\bL\bk_j \,\de\bx  \\
& = \int_{E_{\ep,\ell}}  \sum_{i,j\neq \ell, i<j}  \nabla (\bk_j\cdot \bv)\cdot\bL\bk_i +  
\nabla( \bk_i\cdot \bv)\cdot\bL\bk_j \,\de\bx  \\
& = \sum_{i\neq \ell}\sum_{j\neq i} \int_{E_{\ep,\ell}}   \nabla (\bk_j\cdot \bv)\cdot\bL\bk_i\,\de\bx
 = \sum_{i\neq \ell}\sum_{j\neq i} \int_{\partial E_{\ep,\ell}} (\bk_j\cdot \bv)\cdot\bL\bk_i \cdot \bn\, \de s
\label{term2}
\end{split}
\end{equation}
Between the first and second lines we used $D_\xi \bk_i = \nabla (\bk_i \cdot \bv)$ for $i\ne \ell$
and $D_\xi \bk_\ell = 0$ by \eqref{dkzero},
and in the third line we used ${\rm div}(\bL \bk_i)=0$.

For the second sum of \eqref{split2}, using \eqref{dvlk} from  Lemma \ref{lem:kdot},  we have
\begin{equation}
\begin{split}
\Dvl \left( \sum_{m\ne \ell} \int_{E_{\ep,m}}  \sum_{i<j}  \bk_j\cdot\bL\bk_i\,\de\bx\right)&=
 \sum_{m\ne \ell} \sum_{i\ne \ell} \int_{E_{\ep,m}} \Dvl \bk_\ell \cdot \bL \bk_i \,\de\bx
= -\sum_{m\ne \ell} \sum_{i\ne \ell} \int_{E_{\ep,m}} \nabla( \bk_\ell\cdot \bv) \cdot \bL \bk_i \,\de\bx \\
&= -\sum_{m\ne \ell} \sum_{i\ne \ell} \int_{\partial E_{\ep,m}}(\bk_\ell \cdot\bv )\bL \bk_i \cdot \bn \, \de s
  \label{term6}
\end{split}
\end{equation}
The third term comprising $\overline{U}$ in \eqref{Ubar} is split as
\begin{equation}
  \label{split3}
\sum_{m=1}^N  \int_{E_{\ep,m}} W(\nabla u_0)\,\de\bx = 
\int_{E_{\ep,\ell}} W(\nabla u_0)\,\de\bx + \sum_{m\ne \ell} \int_{E_{\ep,m}} W(\nabla u_0)\,\de\bx.
\end{equation}
To calculate the derivative of the first term on the right side of 
\eqref{split3}, we use \eqref{fder},
but integrate the $D_\xi$ term by parts directly. 
Using $D_\xi u_0 = \Dvl u_0 + \nabla u_0 \cdot \bv$ and ${\rm div} (\bL \nabla u_0)=0$
we have
\begin{equation}
\begin{split}
  \Dvl \left(\int_{E_{\ep,\ell}} W(\nabla u_0)\,\de\bx\right) &= 
\int_{E_{\ep,\ell}}  \nabla (D_\xi u_0)  \bL\nabla u_0 \,\de\bx 
= \int_{\partial E_{\ep,\ell}}  (D_\xi u_0)  \bL\nabla u_0 \cdot \bn\, \de s \\
&= \int_{\partial E_{\ep,\ell}}  (\Dvl u_0 + \nabla u_0 \cdot \bv)  \bL\nabla u_0 \cdot \bn\, \de s .
 \label{term3}
\end{split}
\end{equation}
Calculating the derivative of the second term on the right side of \eqref{split3}
is almost the same as in \eqref{term3} except the domains $E_{\ep,m}$ do not 
depend on $\bz_\ell$ because $m\ne \ell$. Hence
\begin{equation}
  \label{term7}
  \Dvl \left( \sum_{m\ne \ell} \int_{E_{\ep,m}} W(\nabla u_0)\,\de\bx\right) = 
\sum_{m\ne \ell} \int_{E_{\ep,m}} \nabla (\Dvl u_0)\cdot \bL\nabla u_0\,\de\bx
=\sum_{m\ne \ell} \int_{\partial E_{\ep,m}} \Dvl u_0\cdot \bL\nabla u_0\cdot \bn \, \de s.
\end{equation}
Turning to the the final term in \eqref{Ubar}, which we split as
\begin{equation}
  \label{split4}
  \sum_{m=1}^N\sum_{i=1}^N\int_{\partial E_{\ep,m}}u_0 \bL\bk_i\cdot \bn\, \de s=
 \sum_{i=1}^N  \int_{\partial E_{\ep,\ell}}u_0  \bL\bk_i\cdot \bn\, \de s +
\sum_{m\ne \ell}  \int_{\partial E_{\ep,m}}  \sum_{i=1}^N  u_0 \bL\bk_i\cdot \bn\, \de s,
\end{equation}
we calculate the derivative of the first term using \eqref{gder} to get
\begin{equation}
  \label{term4a}
  \Dvl \left(  \sum_{i=1}^N   \int_{\partial E_{\ep,\ell}} u_0 \bL\bk_i\cdot \bn\, \de s\right)=
\sum_{i=1}^N   \int_{\partial E_{\ep,\ell}} (D_\xi u_0) \bL\bk_i\cdot \bn\, \de s
+\sum_{i=1}^N   \int_{\partial E_{\ep,\ell}} u_0 \bL (D_\xi \bk_i)\cdot \bn\, \de s.
\end{equation}
From \eqref{dkzero} we have $D_\xi \bk_\ell = 0$ and from \eqref{dvlk} we have
$D_\xi \bk_i = \nabla (\bk_i \cdot \bv)$ for $i\ne \ell$. Hence, for $i \ne \ell$ we have
\begin{equation}
\begin{split}
 \int_{\partial E_{\ep,\ell}} u_0 \bL (D_\xi \bk_i)\cdot \bn\, \de s & = 
\int_{\partial E_{\ep,\ell}} u_0 \bL \nabla (\bk_i\cdot \bv)\cdot \bn\, \de s
= \int_{ E_{\ep,\ell}} {\rm div}\left( u_0 \bL \nabla (\bk_i\cdot \bv) \right)\,\de\bx\\
& = \int_{ E_{\ep,\ell}} \nabla u_0 \cdot \bL \nabla (\bk_i\cdot \bv)\, \de\bx+
\int_{ E_{\ep,\ell}} u_0 \, {\rm div}\left( \bL \nabla (\bk_i\cdot \bv) \right)\,\de\bx \\
& = \int_{ E_{\ep,\ell}} \nabla (\bk_i\cdot \bv) \cdot \bL \nabla u_0\,\de\bx
=\int_{ E_{\ep,\ell}} (\bk_i\cdot \bv) \cdot \bL \nabla u_0\cdot \bn\, \de s.
  \label{term4b}
\end{split}
\end{equation}
We used $ {\rm div}\left( \bL \nabla (\bk_i\cdot \bv) \right)=(\bv \cdot \nabla)(
{\rm div}(\bL \bk_i))=0$, which follows from $\curl \bk_i = 0$ and ${\rm div}(\bL \bk_i)=0$.
Combining \eqref{term4a} and \eqref{term4b} we get
\begin{equation}
  \label{term4}
  \Dvl \left(  \sum_{i=1}^N   \int_{\partial E_{\ep,\ell}} u_0 \bL\bk_i\cdot \bn\, \de s\right)=
\sum_{i=1}^N   \int_{\partial E_{\ep,\ell}} (D_\xi u_0) \bL\bk_i\cdot \bn\, \de s
+\sum_{i\ne \ell}  \int_{\partial E_{\ep,\ell}} (\bk_i \cdot \bv) \bL  \nabla  u_0  \cdot \bn\, \de s
\end{equation}
Finally, the derivative of the second term in \eqref{split4} is calculated similarly to the
first, but is simpler because the domains of integration are independent of $\bz_\ell$. Hence,
\begin{equation}
  \Dvl \left(  \sum_{m\ne \ell}  \int_{\partial E_{\ep,m}}  \sum_{i=1}^N  u_0 \bL\bk_i\cdot \bn \,\de s \right)
=  \sum_{m\ne \ell}  \int_{\partial E_{\ep,m}}  \!\!\! u_0 \bL (\Dvl \bk_\ell) \cdot \bn \,\de s+
 \sum_{m\ne \ell}   \sum_{i=1}^N  \int_{\partial E_{\ep,m}} \!\!\!  (\Dvl u_0) \bL\bk_i\cdot \bn\, \de s
  \label{term8a}
\end{equation}
because $\Dvl \bk_i = 0$ when $i\ne \ell$. 
Using $ {\rm div}\left( \bL \nabla (\bk_i\cdot \bv) \right) =0 $, as we did
to get \eqref{term4b}, we have
\begin{equation}
\begin{split}
  \int_{\partial E_{\ep,m}} u_0 \bL (\Dvl \bk_\ell) \cdot \bn\, \de s &= 
-\int_{\partial E_{\ep,m}} u_0 \bL \nabla( \bk_\ell\cdot \bv) \cdot \bn\,\de s = 
-\int_{E_{\ep,m}} {\rm div}(u_0 \bL \nabla( \bk_\ell\cdot \bv))\, \de\bx\\
&=-\int_{E_{\ep,m}}  \nabla u_0 \bL \nabla( \bk_\ell\cdot \bv)\, \de\bx
=-\int_{\partial E_{\ep,m}}  ( \bk_\ell\cdot \bv) \bL\nabla u_0 \cdot \bn\, \de s 
\label{term8b}
  \end{split}
\end{equation}
Then \eqref{term8a} and \eqref{term8b} give
\begin{equation}
  \label{term8}
  \Dvl \left(  \sum_{m=1}^N\sum_{i=1}^N\int_{\partial E_{\ep,m}}u_0 \bL\bk_i\cdot \bn\, \de s \right)=
\sum_{m\ne \ell} \int_{\partial E_{\ep,m}} \left(\sum_{i=1}^N \Dvl u_0 \cdot \bL \bk_i \cdot  \bn
-(\bk_\ell \cdot \bv) \bL \nabla u_0 \cdot \bn \right)\de s
\end{equation}
Combining \eqref{term1}, \eqref{term2}, \eqref{term3}, and \eqref{term4} we have
\begin{equation}
\begin{split}
\int_{\partial E_{\ep,\ell}}  &\left\{ \sum_{i\ne \ell}   (\bk_i \cdot \bv) \bL \bk_i\cdot \bn + 
\sum_{i\neq \ell}\sum_{j\neq i}  (\bk_j\cdot \bv)\cdot\bL\bk_i \cdot \bn \,\,+ \right. \\
 & \left.  + \,\, (D_\xi u_0)  \bL\nabla u_0 \cdot \bn +
\sum_{i=1}^N   (D_\xi u_0) \bL\bk_i\cdot \bn
+\sum_{i\ne \ell}  (\bk_i \cdot \bv) \bL  \nabla  u_0  \cdot \bn \right\}\de s \\
& = \int_{\partial E_{\ep,\ell}}  \left\{
 \sum_{i\ne \ell}   (\bk_i \cdot \bv)\left(\bL \nabla u_0 + \sum_{j=1}^N\bL\bk_j \right) 
+  D_\xi u_0 \left( \bL \nabla u_0 + \sum_{j=1}^N\bL \bk_j \right) \right\}\cdot \bn\,\de s\\
&=  \int_{\partial E_{\ep,\ell}} \left( D_\xi u_0 + \sum_{i\ne \ell}    \bk_i \cdot \bv \right) \bL \bh_0 \cdot 
\bn\, \de s
  \label{comb1}
\end{split}
\end{equation}
Combining \eqref{term5}, \eqref{term6}, \eqref{term7}, and \eqref{term8} we have
\begin{align}
\sum_{m\ne \ell} \int_{\partial E_{\ep,m}} & \left\{
-(\bk_\ell \cdot \bv) \bL \bk_\ell 
- \sum_{i\ne \ell}  (\bk_\ell \cdot\bv )\bL \bk_i + 
   \Dvl u_0 \cdot \bL \nabla u_0 
 + \sum_{i=1}^N \Dvl u_0 \cdot \bL \bk_i  
-(\bk_\ell \cdot \bv) \bL \nabla u_0
\right\} \cdot \bn \,\de s \nonumber \\
  & =\sum_{m\ne \ell} \int_{\partial E_{\ep,m}} 
\left( \Dvl u_0 - \bk_\ell \cdot \bv\right) \bL \bh_0 \cdot \bn
\,\de s. \label{comb2}
\end{align}
Thus, \eqref{DUhat}, \eqref{comb1}, and \eqref{comb2} together give
\begin{equation*}
D_{\bz_\ell}U (\bv)=
  \nabla_{\bz_\ell} U \cdot \bv
= \left(\nabla_{\bz_\ell} \widehat{U} + \nabla_{\bz_\ell} \overline{U}\right)\cdot \bv = 
-\int_{\partial E_{\ep,\ell}} \left\{ W(\bh_0){\bf I} - \bh_0 
 \otimes (\bL\bh_0) \right\}\bn \, \de s \cdot \bv,
\end{equation*}
which establishes \eqref{gradU}.
\end{proof}

\noindent {\bf Acknowledgments.} 
The authors warmly thank the Center for Nonlinear Analysis (NSF Grant No.
DMS-0635983), where part of this research was carried out. 
The research of T.\@ Blass was supported by the National Science
Foundation under the PIRE Grant No.\@ OISE-0967140. 
M.\@ Morandotti also acknowledges support of the 
Funda\c{c}\~{a}o para a Ci\^encia e a Tecnologia 
(Portuguese Foundation for Science and Technology)
through the Carnegie Mellon Portugal Program 
under the grant FCT$\_$UTA/CMU/MAT/0005/2009.


\bibliographystyle{plain}
\bibliography{disloc}

\end{document}